\documentclass[table]{amsart}
\usepackage[margin=1.1in]{geometry}
\usepackage{amscd,amsmath,amsxtra,amsthm,amssymb,stmaryrd,xr,mathrsfs,mathtools,enumerate,commath, comment, enumitem, graphicx}
\usepackage{stmaryrd}
\usepackage{multirow}
\usepackage{xcolor}
\usepackage{commath}
\usepackage{comment}
\usepackage{tikz-cd}
\usepackage{tkz-graph}

\usepackage{longtable} 
\usepackage{pdflscape} 
\usepackage{booktabs}
\usepackage{hyperref}
\definecolor{vegasgold}{rgb}{0.77, 0.7, 0.35}
\definecolor{darkgoldenrod}{rgb}{0.72, 0.53, 0.04}
\definecolor{gold(metallic)}{rgb}{0.83, 0.69, 0.22}
\hypersetup{
 colorlinks=true,
 linkcolor=darkgoldenrod,
 filecolor=brown,      
 urlcolor=gold(metallic),
 citecolor=darkgoldenrod,
 }

\usepackage[all,cmtip]{xy}

\DeclareFontFamily{U}{wncy}{}
\DeclareFontShape{U}{wncy}{m}{n}{<->wncyr10}{}
\DeclareSymbolFont{mcy}{U}{wncy}{m}{n}
\DeclareMathSymbol{\Sh}{\mathord}{mcy}{"58}
\usepackage[T2A,T1]{fontenc}
\usepackage[OT2,T1]{fontenc}

\usepackage{tikz}
\usetikzlibrary{shapes.geometric}
\usetikzlibrary{decorations.markings}
\tikzset{every loop/.style={min distance=10mm,looseness=10}}
\tikzstyle{vertex}=[auto=left,circle,minimum size=1pt,inner sep=0pt]

\newtheorem{theorem}{Theorem}[section]
\newtheorem{lemma}[theorem]{Lemma}

\newtheorem*{theorem*}{Theorem}
\newtheorem*{ass*}{Assumption}
\newtheorem{definition}[theorem]{Definition}
\newtheorem{corollary}[theorem]{Corollary}

\newtheorem{proposition}[theorem]{Proposition}

\newcommand{\sX}{\mathscr{X}}

\newcommand{\cK}{\mathcal{K}}

\newcommand{\Z}{\mathbb{Z}}

\newcommand{\Q}{\mathbb{Q}}
\newcommand{\F}{\mathbb{F}}

\newcommand{\cO}{\mathcal{O}}

\newcommand{\Hom}{\mathrm{Hom}}

\newcommand{\op}[1]{\operatorname{#1}}

\numberwithin{equation}{section}

\begin{document}

\title[Iwasawa theory and the representations of finite groups]{Iwasawa theory and the representations of finite groups}

\author[A.~Ray]{Anwesh Ray}
\address[Ray]{Chennai Mathematical Institute, H1, SIPCOT IT Park, Kelambakkam, Siruseri, Tamil Nadu 603103, India}
\email{anwesh@cmi.ac.in}


\maketitle

\begin{abstract}
In this note, I develop a representation-theoretic refinement of the Iwasawa theory of finite Cayley graphs. Building on analogies between graph zeta functions and number-theoretic \(L\)-functions, I study \(\mathbb{Z}_\ell\)-towers of Cayley graphs and the asymptotic growth of their Jacobians. My main result establishes that the Iwasawa polynomial associated to such a tower admits a canonical factorization indexed by the irreducible representations of the underlying group. This leads to the definition of representation-theoretic Iwasawa polynomials, whose properties are studied.
\end{abstract}

\section{Introduction}
\par The theory of zeta and \( L \)-functions associated to graphs, originating in the seminal work of Ihara, has revealed striking analogies between the spectral theory of finite graphs and the arithmetic of global fields. Chief among these analogies is the existence of graph zeta functions, which—much like their number-theoretic counterparts—admit Euler product factorizations akin to the Artin formalism. Moreover, the special values of these zeta functions encode graph-theoretic invariants that serve as combinatorial analogues of class numbers, regulators, and other arithmetic quantities. For a comprehensive introduction to the subject, the reader is referred to \cite{Terras:2011}.

\par In recent years, a new direction in this field has emerged with the development of an Iwasawa-theoretic perspective on graphs. This was initiated independently by Vallières \cite{Vallieres:2021} and Gonet \cite{Gonet:2021a, Gonet:2022}, who introduced the notion of \(\mathbb{Z}_\ell\)-towers of finite multigraphs and established analogues of Iwasawa’s classical results on the asymptotic growth of arithmetic invariants. In this graph-theoretic setting, the complexity of a graph—measured as the cardinality of its Jacobian or sandpile group—plays the role of the class number. Along such towers, one observes analogues of the classical \(\mu\)-, \(\lambda\)-, and \(\nu\)-invariants, as well as a natural graph-theoretic analogue of the Iwasawa polynomial. This new Iwasawa theory of graphs has rapidly attracted attention, leading to a series of further developments and refinements, as seen in \cite{mcgownvallieresII, mcgownvallieresIII, DuBose/Vallieres:2022, Kleine/Muller:2022, Ray/Vallieres:2022, DLRV, LeiMuller}.

\par In \cite{GhoshRay}, the Iwasawa theory of Cayley graphs associated to finite abelian groups was studied. The present work extends this framework to nonabelian groups. A key result of this paper is that the Iwasawa polynomial associated to a tower of Cayley graphs admits a canonical factorization (see Theorem \ref{main thm decomp}), with each factor corresponding to an irreducible representation of the underlying finite group. This observation leads naturally to the definition of a representation-theoretic Iwasawa polynomial associated to any irreducible representation of a finite group, possibly nonabelian. I study some of the properties of these representation-theoretic Iwasawa polynomials and their associated invariants, exploring their structural properties and their behavior with respect to congruences. I illustrate my results via an illustrative example. 

\subsection*{Acknowledgment} The author thanks Katharina M\"uller for helpful suggestions.

\section{Preliminary notions}
\subsection{Galois theory of graphs and the Artin Ihara L-functions}
\par The interplay between spectral graph theory and number theory has become increasingly rich and intricate. At the heart of this interaction lies the theory of Ihara zeta functions of graphs, which are combinatorial analogues of the Dedekind zeta functions of number fields. These zeta functions encode spectral and topological data about the graph, and admit explicit determinant expressions reminiscent of the functional equations satisfied by arithmetic L-functions. These considerations will lead me to interesting new connections with the representation theory of groups. In the next section, I lay the groundwork for a systematic study of towers of graph covers and their associated $\ell$-adic
L-functions, in the spirit of Iwasawa theory. In due course, I shall specialize my discussion to Cayley graphs associated to groups.

The graphs in this article are finite, undirected, and without loops. A graph $\sX$ is a quadruple $(V_\sX, E_\sX^+, i, \iota)$, where $V_\sX = {v_1, \dots, v_n}$ is the vertex set, $E_\sX^+$ the set of oriented edges, $i: E_\sX^+ \to V_\sX \times V_\sX$ the incidence map, and $\iota: E_\sX^+ \to E_\sX^+$ the edge inversion satisfying $i \circ \iota = \tau \circ i$, with $\tau(v, v') = (v', v)$. For convenience, I shall occasionally write $\bar{e}:=\iota(e)$. An edge $e \in E_\sX^+$ joins $v$ to $v'$ when $i(e) = (v, v')$, and $\iota(e)$ joins $v'$ to $v$. I write $e \sim e'$ if $e' = \iota(e)$, and denote by $E_\sX$ the set of equivalence classes under this relation. Thus, $E_\sX^+$ represents directed edges and $E_\sX$ the corresponding undirected ones. The natural projection $\pi: E_\sX^+ \to E_\sX$ sends $e$ to its class. Define the incidence matrix $A_{\sX}=(a_{i,j})$ of the graph $\sX$, where $a_{i,j}$ is the number of edges from $v_i$ to $v_j$. The source and target maps $o, t: E_{\sX}^+\rightarrow V_{\sX}$ are the compositions of $i$ with the projections to the first and second factor of $V_{\sX}\times V_{\sX}$ respectively. The \emph{degree} of $v$ is defined as the number of edges emanating from $v$, i.e., $\op{deg}(v):=\# E_{\sX,v}^+$. The betti numbers of $\sX$ are defined as follows 
\[b_i(\sX):=\op{rank}_{\Z} H_i(\sX,\Z).\] The \emph{Euler characteristic} is defined as follows $\chi(\sX):=b_0(\sX)-b_1(\sX)$. When $\sX$ is connected, $b_0(\sX)=1$ and $b_1(\sX)=\# E_{\sX}-\# V_{\sX}+1$. One has that $\chi(\sX)=\# V_{\sX}-\# E_{\sX}$. It will be assumed throughout that all my multigraphs are connected with no vertices having degree equal to $1$. Moreover, I assume that $\chi(\sX)\neq 0$, i.e., the graph is not a cycle graph.

\par A graph can be viewed, to some extent, as a discrete analogue of a Riemann surface. For instance, there are graph theoretic analogues of a Jacobian, and the Riemann--Roch theorem \cite{BakerNorine}. Let me introduce some basic definitions which will be of use in this article. The divisor group $\op{Div}(\sX)$ is the free abelian group on the vertices $V_\sX$, consisting of formal sums $D = \sum_v n_v v$ with $n_v \in \Z$. The degree map $\deg: \op{Div}(\sX) \to \Z$, given by $\deg(D) = \sum_v n_v$, has kernel $\op{Div}^0(\sX)$. Let $\mathcal{M}(\sX)$ be the group of $\Z$-valued functions on $V_\sX$, freely generated by the characteristic functions $\chi_v$. The map \[\op{div}: \mathcal{M}(\sX) \to \op{Div}^0(\sX)\] is defined by setting $\op{div}(\chi_v) = \sum_w \rho_w(v) w$, where \[
\rho_w(v) = 
\begin{cases}
\op{val}_{\sX}(v) - 2 \cdot \#\text{loops at } v & \text{if } w = v, \\
-\#\text{edges from } w \text{ to } v & \text{if } w \ne v.
\end{cases}
\] Extending linearly, one obtains for $f \in \mathcal{M}(\sX)$ the formula $\op{div}(f) = -\sum_v m_v(f) \cdot v$, where $m_v(f) := \sum_{e \in E_{\sX,v}^+} (f(t(e)) - f(o(e)))$.
The image $\op{Pr}(\sX)$ of $\op{div}$ is the group of principal divisors, and the quotient $\op{Pic}^0(\sX) := \op{Div}^0(\sX)/\op{Pr}(\sX)$ is the Jacobian of $\sX$. Its cardinality $\kappa_{\sX} := \#\op{Pic}^0(\sX)$ is called the complexity of $\sX$, analogous to the class number of a number ring (see \cite{divisorsandsandpiles}).

A morphism of graphs $f: \mathscr{Y} \to \sX$ consists of functions $f_V: V_{\mathscr{Y}} \to V_{\sX}$, $f_E: E_{\mathscr{Y}}^+ \to E_{\sX}^+$ such that
\[
f_V(o(e)) = o(f_E(e)), \quad f_V(t(e)) = t(f_E(e)), \quad f_E(\iota(e)) = \iota(f_E(e)).
\]
It is a cover if $f_V$ is surjective and for each $w \in V_{\mathscr{Y}}$, the map $f: E_{\mathscr{Y},w}^+ \to E_{\sX,f(w)}^+$ is a bijection. A cover is Galois if $\mathscr{Y}$ and $\sX$ are connected and the group $\op{Aut}_f(\mathscr{Y}/\sX)$ acts transitively on each fiber $f^{-1}(v)$. I write $\op{Gal}(\mathscr{Y}/\sX) := \op{Aut}_f(\mathscr{Y}/\sX)$.

To define Artin--Ihara $L$-functions, let $\mathfrak{c} = a_1 \dots a_k$ be a walk in $\sX$, where $t(a_i) = o(a_{i+1})$. Such a walk is a cycle if $o(a_1) = t(a_k)$, and is prime if it has no backtracks or tails and is not a nontrivial power of a shorter cycle. For a Galois cover $\mathscr{Y}/\sX$ with abelian Galois group $G$, and character $\psi \in \widehat{G} := \Hom(G, \mathbb{C}^\times)$, the Artin--Ihara $L$-function is defined by
\[
L_{\mathscr{Y}/\sX}(u, \psi) := \prod_{\mathfrak{c}} \left(1 - \psi\left(\left(\frac{\mathscr{Y}/\sX}{\mathfrak{c}}\right)\right) u^{l(\mathfrak{c})} \right)^{-1},
\]
where the product runs over all prime cycles $\mathfrak{c}$ in $\sX$, and $\left(\frac{\mathscr{Y}/\sX}{\mathfrak{c}}\right)$ denotes the Frobenius automorphism associated to $\mathfrak{c}$ (cf. \cite[Definition 16.1]{Terras:2011}). The special case $\psi = 1$ and $\mathscr{Y} = \sX$ recovers the Ihara zeta function $\zeta_{\sX}(u)$.

Consider an abelian cover $ \mathscr{Y} \to \sX $ with Galois group $ G = \operatorname{Aut}(\mathscr{Y}/\sX) $. For each $ i = 1, \dots, g_\sX $, fix a vertex $ w_i $ in the fiber above $ v_i \in V_\sX $. For $ \sigma \in G $, define the matrix $ A(\sigma) = (a_{i,j}(\sigma)) $ by
\[
a_{i,j}(\sigma) =
\begin{cases}
2 \times \text{(number of loops at } w_i), & \text{if } i = j \text{ and } \sigma = 1; \\
\text{number of edges from } w_i \text{ to } w_j^\sigma, & \text{otherwise}.
\end{cases}
\]
For each character $ \psi \in \widehat{G} $, define the twisted adjacency matrix
\[
A_\psi := \sum_{\sigma \in G} \psi(\sigma) A(\sigma).
\]
Let $ D = \operatorname{diag}(\deg(v_1), \dots, \deg(v_{g_\sX})) $. Then, the Artin-Ihara L-function is given by
\[
L_{\mathscr{Y}/\sX}(u, \psi)^{-1} = (1 - u^2)^{-\chi(\sX)} \cdot \det(I - A_\psi u + (D - I) u^2),
\]
cf. \cite[Theorem 18.15]{Terras:2011}. Set
\[
h_\sX(u, \psi) := \det(I - A_\psi u + (D - I) u^2), \quad h_\sX(u) := h_\sX(u, 1).
\]

The following result links the derivative of $ h_\sX $ at 1 to the complexity $ \kappa_\sX $ of the graph:

\begin{theorem}[\cite{Northshield}, \cite{HammerMattmanSandsVallieres}]\label{class number formula thm}
Assume that $\sX $ is connected and $\chi(\sX)\neq 0$, then, $
h_\sX'(1) = -2\chi(\sX)\kappa_\sX$.
\end{theorem}
This result is strikingly parallel to that of classical class number formulas in number theory, where the zeta function of an extension factors over characters of the Galois group, and the special values encode arithmetic invariants such as regulators and class numbers. Artin formalism gives a factorization of zeta functions of covers:

\begin{theorem}\label{artin formalism thm}
If $ \mathscr{Y} \to \sX $ is an abelian Galois cover with group $ G $, then
\[
\zeta_{\mathscr{Y}}(u) = \zeta_\sX(u) \cdot \prod_{\substack{\psi \in \widehat{G} \\ \psi \neq 1}} L_{\mathscr{Y}/\sX}(u, \psi).
\]
\end{theorem}
\begin{proof}
    For a proof, I refer to \cite{Terras:2011}.
\end{proof}

Evaluating at $ u = 1 $, I obtain a relation between complexities:

\begin{corollary}\label{class number relation corollary}
Under the same assumptions, one has:
\[
|G|\kappa_{\mathscr{Y}} = \kappa_{\sX} \prod_{\substack{\psi \in \widehat{G} \\ \psi \neq 1}} h_{\sX}(1, \psi).
\]
\end{corollary}
The above result implies in particular that each $ h_\sX(1, \psi) \neq 0 $ for nontrivial $ \psi \in \widehat{G} $.

\subsection{Iwasawa theory of graphs}
\par In this section, I discuss the Iwasawa theory of $\Z_\ell$-towers over a connected graph $\sX$ for which it is assumed throughout that $\chi(\sX)\neq 0$. I begin by explaining how certain Galois covers of $\sX$ may be constructed from combinatorial data known as \emph{voltage assignments}. Let $\sX$ be a graph, and let $\pi: E_{\sX}^+ \to E_{\sX}$ denote the natural projection from the set of oriented edges to the set of unoriented edges, which associates to each oriented edge its underlying unoriented edge. Fix a section $\gamma: E_{\sX} \to E_{\sX}^+$ of $\pi$, so that each unoriented edge is assigned a distinguished orientation. Setting $S:=\gamma(E_\sX)$, a \emph{voltage assignment} is a function $\alpha:S \to G$. I extend the voltage assignment $\alpha$ to all of $E_\sX^+$ by declaring $\alpha(\bar{e}) = \alpha(e)^{-1}$ for every $e \in E_\sX^+$. Given this data, one constructs a multigraph $\sX(G, S, \alpha)$ as follows. The vertex set is $V = V_\sX \times G$, and the set of directed edges is $E^+ = E_\sX^+ \times G$. Each directed edge $(e, \sigma) \in E^+$ connects the vertex $(o(e), \sigma)$ to the vertex $(t(e), \sigma \cdot \alpha(e))$, where $o(e)$ and $t(e)$ denote the origin and target of the edge $e$, respectively. The edge-reversal map is defined by  
\[
\overline{(e, \sigma)} = (\bar{e}, \sigma \cdot \alpha(e)).
\]

Now suppose $G_1$ is another finite abelian group, and let $f: G \to G_1$ be a group homomorphism. Then $f$ induces a morphism of multigraphs  
\[
f_* : \sX(G, S, \alpha) \to \sX(G_1, S, f \circ \alpha),
\]
defined on vertices and edges by  
\[
f_*(v, \sigma) = (v, f(\sigma)) \quad \text{and} \quad f_*(e, \sigma) = (e, f(\sigma)).
\]
\begin{definition} \label{tower}
Let $\ell$ be a prime, and let \( \sX \) be a connected graph. A \( \mathbb{Z}_\ell \)-tower over \( \sX \) is a sequence of connected graph covers
\[
\sX = \sX_0 \longleftarrow \sX_1 \longleftarrow \sX_2 \longleftarrow \cdots
\]
such that for each \( n \geq 1 \), the composite cover \( \sX_n \to \sX \) is Galois with Galois group isomorphic to \( \mathbb{Z}/\ell^n\mathbb{Z} \).
\end{definition}

I now describe a natural way to construct such towers using voltage assignments. Fix a finite set \( S \) of oriented edges of the base graph \( \sX \), and let  
\[
\alpha = (\alpha_1, \alpha_2, \dots, \alpha_t) \in \mathbb{Z}_\ell^t,
\]
where \( t = |S| \), and each \( \alpha_i = \alpha(s_i) \) for a chosen enumeration \( S = \{s_1, \dots, s_t\} \). The map \( \alpha \) may be interpreted as a continuous homomorphism from the free abelian group on \( S \) into \( \mathbb{Z}_\ell \), i.e., a \(\mathbb{Z}_\ell\)-valued voltage assignment.

For each \( n \geq 1 \), let \( \alpha_{/n} \) denote the reduction of \( \alpha \) modulo \( \ell^n \), taking values in \( \mathbb{Z}/\ell^n\mathbb{Z} \). Applying the voltage graph construction to \( \alpha_{/n} \), I obtain a sequence of finite Galois covers
\[
\sX(\mathbb{Z}/\ell^n\mathbb{Z}, S, \alpha_{/n}) \longrightarrow \sX.
\]
These fit into a tower:
\[
\sX \longleftarrow \sX(\mathbb{Z}/\ell\mathbb{Z}, S, \alpha_{/1}) \longleftarrow \sX(\mathbb{Z}/\ell^2\mathbb{Z}, S, \alpha_{/2}) \longleftarrow \cdots,
\]
which defines a \( \mathbb{Z}_\ell \)-tower over \( \sX \) in the sense of Definition~\ref{tower}.

\par In the following discussion, I assume that the multigraphs \( \sX(\mathbb{Z}/\ell^n\mathbb{Z}, S, \alpha_{/n}) \) are connected for all \( n \geq 0 \). An explicit condition ensuring the connectedness of such graphs can be described in terms of the fundamental group. Given a walk \( w = a_1 a_2 \dots a_n \) in \( \sX \), I define the product \( \alpha(w) := \alpha(e_1) \cdots \alpha(e_n) \in G \), where \( \alpha: S \to G \) satisfies \( \alpha(\iota(e)) = \alpha(e)^{-1} \). It follows that homotopically equivalent walks \( c_1 \) and \( c_2 \) have equal image under \( \alpha \). Fixing a base vertex \( v_0 \in V_\sX \), the map \( \alpha \) induces a group homomorphism \( \rho_\alpha : \pi_1(\sX, v_0) \to G \) defined by \( \rho_\alpha([\gamma]) = \alpha(\gamma) \). When \( \sX \) is connected, the derived graph \( \sX(G, S, \alpha) \) is connected if and only if \( \rho_\alpha \) is surjective; this equivalence is established in \cite[Theorem 2.11]{Ray/Vallieres:2022}.

\par Now, suppose \( \sX \) is a connected graph with vertex set \( \{v_1, \dots, v_{g_\sX}\} \). Define the matrix \( D_\sX = (d_{i,j}) \) where \( d_{i,j} = \deg(v_i) \) if \( i = j \) and \( 0 \) otherwise. The Laplacian matrix is \( Q_\sX := D_\sX - A_\sX \), with \( A_\sX \) the adjacency matrix. Let \( \alpha: S \to \mathbb{Z}_\ell \) be a voltage assignment satisfying \( \alpha(\iota(e)) = -\alpha(e) \). This extends to a matrix \[M(x) = M_{\sX,\alpha}(x) \in \mathbb{Z}_\ell[x;\mathbb{Z}_\ell]^{g_\sX \times g_\sX} ,\] defined by subtracting from \( D_\sX \) the matrix whose \((i,j)\)-entry is \[\sum_{e \in E_\sX^+, i(e) = (v_i, v_j)} x^{\alpha(e)}.\] Here, \( \mathbb{Z}_\ell[x; \mathbb{Z}_\ell] \) consists of expressions \( \sum_a c_a x^a \) with \( a \in \mathbb{Z}_\ell \) and \( c_a \in \mathbb{Z}_\ell \). The Iwasawa polynomial associated to the tower defined by \( \alpha \) is \( f_{\sX,\alpha}(T) := \det M(1+T) \in \mathbb{Z}_\ell\llbracket T \rrbracket \). Though not necessarily a polynomial, this formal power series becomes a polynomial after multiplying by a suitable power of \( (1+T) \). For the tower of derived graphs  
\[
\sX \leftarrow \sX(\mathbb{Z}/\ell\mathbb{Z}, S, \alpha_{/1}) \leftarrow \sX(\mathbb{Z}/\ell^2\mathbb{Z}, S, \alpha_{/2}) \leftarrow \dots,
\]
the evaluation \( f_{\sX,\alpha}(1 - \zeta_{\ell^n}) = h_\sX(1, \psi_n) \) for any primitive \( \ell^n \)-th root of unity \( \zeta_{\ell^n} \) and character \( \psi_n : \mathbb{Z}/\ell^n\mathbb{Z} \to \mathbb{C} \) defined by \( \psi_n(\bar{1}) = \zeta_{\ell^n} \), as shown in \cite[Corollary 5.6]{mcgownvallieresIII}. Since \( Q_\sX \) is singular with \( u = (1, 1, \dots, 1)^t \) in its kernel, I deduce that \( f_{\sX,\alpha}(0) = \det Q_\sX = 0 \), and thus \( T \) divides \( f_{\sX,\alpha}(T) \). Consequently, \[f_{\sX,\alpha}(T) = T g_{\sX,\alpha}(T),\] where \( g_{\sX,\alpha}(T) \in \mathbb{Z}_\ell\llbracket T \rrbracket \) is a power series and \( m \in \mathbb{Z}_{\geq 0} \) is minimal such that \( g_{\sX,\alpha}(T) \) becomes a polynomial. By the \( \ell \)-adic Weierstrass Preparation Theorem, there exists a factorization \( g_{\sX,\alpha}(T) = \ell^\mu P(T) u(T) \), where \( P(T) \in \mathbb{Z}_\ell[T] \) is a distinguished polynomial and \( u(T) \in \mathbb{Z}_\ell\llbracket T \rrbracket \) is a unit, i.e., \( u(0) \in \mathbb{Z}_\ell^\times \). The Iwasawa invariants associated to the tower are defined as \( \mu_\ell(\sX,\alpha) := \mu \) and \( \lambda_\ell(\sX,\alpha) := \deg P(T) \). Finally, a powerful result of Gonet \cite{Gonet:2021a, Gonet:2022}, Vallieres \cite{Vallieres:2021}, and McGown--Vallieres \cite{mcgownvallieresII,mcgownvallieresIII} states that if \( \alpha \) is a voltage assignment satisfying the assumptions above (including my connectivity assumption), then for \( n \gg 0 \), the complexity \( \kappa_\ell(X_n) \) of the derived graph \( X_n := \sX(\mathbb{Z}/\ell^n\mathbb{Z}, S, \alpha_{/n}) \) satisfies the formula \[\kappa_\ell(X_n) = \ell^{\ell^n \mu + n \lambda + \nu} \] for some integer \( \nu \), as proven in \cite[Theorem 6.1]{mcgownvallieresIII}.

\section{Factorization of the Iwasawa polynomial}
\par In this section, $G$ will be a finite group and $S$ is a subset of $G$ such that:
\begin{itemize}
    \item $g S g^{-1}=S$ for all $g\in G$,
    \item $S$ generates $G$, 
    \item $S=S^{-1}$, and,
    \item $1\notin S$. 
\end{itemize}
Let $\sX$ be the Cayley graph $\op{Cay}(G, S)$ associated with the pair $(G, S)$ and assume throughout that $\chi(\sX)\neq 0$, i.e., that $\sX$ is not a cycle graph. Enumerate $G=\{g_1, \dots, g_n\}$ and write $V_{\sX}=\{v_1, \dots, v_n\}$ where $v_i=v_{g_i}$ is the vertex associated to $g_i$. Set $r:=\# S$, there is an edge $e_{i,j}$ joining $v_i$ to $v_j$ if $g_i g_j^{-1}\in S$. Note that since $S=S^{-1}$, $\sX$ is an undirected graph, and since $1\notin S$, $\sX$ has no loops. Since $S$ generates $G$, it follows that there is a walk from $1$ to any other vertex in $V_{\sX}$, thus, $\sX$ is connected. There is a natural action of $G$ on $\sX$, i.e., a natural group homomorphism:
\[\rho: G\rightarrow \op{Aut}(\sX)\] where $g\in G$ sends $v_{h}$ to $v_{gh}$ and the edge $e$ joining $v_i$ to $v_j$ to the edge $g(e)$, which joins $g(v_i)$ to $g(v_j)$. This action is well defined since $S$ is stable with respect to conjugation.

\begin{definition}\label{def of beta}
    I shall consider voltage assignments that arise from functions on $S$. Let $\ell$ be a prime number and $\beta: S\rightarrow \Z_\ell$ be a function such that: 
\begin{enumerate}
\item $\beta(g a g^{-1})=\beta(a)$,
\item the image of $\beta$ generates $\Z_{\ell}$ (as a $\Z_\ell$-module), 
\item $\beta(s^{-1})=-\beta(s)$ and $\beta(1_G)=0$,
\item the image of $\beta$ lies in $\Z$,
\item there exists $m>0$ and a tuple $(h_1, \dots, h_m)\in S^m$ such that $h_1 h_2 \dots h_m\in S$ and
\begin{equation}\label{congruence equation}
\beta(h_1 h_2 \dots h_m)\not \equiv \sum_{i=1}^m \beta(h_i)\pmod{\ell}.
\end{equation}
\end{enumerate} 
I define a $\Z_\ell$-valued \emph{voltage assignment} $\alpha=\alpha_\beta: E_\sX^+\rightarrow \Z_\ell$ by $\alpha\left(e\right):=\beta(g_1 g_2^{-1})$ where $e$ is the edge joining $v_{g_1}$ to $v_{g_2}$.
\end{definition} 
I choose an ordering and write $G=\{g_1, \dots, g_n\}$ and set $v_i:=v_{g_i}$. For $g\in G$, set \[\delta_S(g):=\begin{cases}
1 &\text{ if }g\in S;\\ 
0 &\text{ if }g\notin S.
\end{cases}\]Recall that a voltage assignment $\alpha: E_\sX^+\rightarrow \Z_\ell$ gives rise to a $\Z_\ell$-tower over $\sX$. It follows from \cite[Proposition 4.3]{GhoshRay} that this tower consists of connected graphs. One has that\[f_{\sX, \alpha}(T)=\op{det}\left(\op{M}_{\sX, \alpha}(1+T)\right)=\op{det}\left(r- \delta_S(g_i g_j^{-1})(1+T)^{\beta(g_i g_j^{-1})}\right)_{i,j}\] is the associated Iwasawa polynomial.

\par Choose an embedding of $\bar{\Q}\hookrightarrow \bar{\Q}_\ell$ and let $\cK$ be a finite extension of $\Q_\ell$ which contains all $n$-th roots of unity. Let $\cO$ denote the valuation ring of $\cK$, $\varpi$ be its uniformizer and $\kappa:=\cO/(\varpi)$ the residue field. $F:=\cK((T))$ and $A:=\bar{F}[G]$ be the group algebra of $G$ over $\bar{F}$. Denote by $\op{Irr}(G)$ the set of irreducible characters $\chi: G\rightarrow \bar{\Q}_\ell$. I note that such characters take values in $\cO$, since they are expressible as sums of $n$-th roots of $1$. For $\chi\in \op{Irr}(G)$, set:
\[Q_\chi(T)=r\chi(1)-\sum_{t\in S} (1+T)^{\beta(t)}\chi(t)\in \cO\llbracket T \rrbracket\] and $P_{\chi}(T):=\frac{Q_\chi(T)}{\chi(1)}$.

\begin{theorem}\label{main thm decomp}
    Let $\sX$ be a Cayley graph associated to the pair $(G,S)$ and $\beta$ satisfy the conditions of Definition \ref{def of beta}. Then, there is a factorization:
    \[f_{\sX, \alpha}(T)=\prod_{\chi\in \op{Irr}(G)} P_\chi(T)^{\chi(1)^2}.\]
\end{theorem}
\begin{proof}
For each \( a \in A \), define a right multiplication operator \( \rho_a: A \to A \) by  
\[
\rho_a(g) := g a \quad \text{for } g \in G.
\]  
I define the adjacency operator \( \operatorname{ad}: A \to A \) by  
\[
\operatorname{ad}(g) := \sum_{t \in S} x^{\beta(t)} \, t g,
\]  
where \( S \subseteq G \) is a fixed subset and \( \beta: S \to \mathbb{Z} \) is a weight function. Evaluating this at the identity element \( 1 \in G \), I obtain  
\[
z := \operatorname{ad}(1) = \sum_{t \in S} x^{\beta(t)} t.
\]  
This element \( z \in A \) lies in the center of \( A \), and since \( \operatorname{ad} \) is defined via left multiplication by \( z \), I deduce that  
\[
\operatorname{ad} = \rho_z.
\]

Decompose the semisimple algebra \( A \) as a direct sum of simple two-sided ideals:  
\[
A = A_1 \oplus \cdots \oplus A_s.
\]  
Since \( z \) is central, it acts on each simple ideal \( A_i \) by multiplication by a scalar \( \lambda_i \). Let \( e_i \in A_i \) denote the identity element of the ideal \( A_i \), viewed as a central idempotent of \( A \). Then I may write  
\[
z = \sum_{i=1}^s \lambda_i e_i.
\]  
It follows that the eigenvalues of \( \operatorname{ad} \) (viewed as a linear operator on \( A \)) are exactly the \( \lambda_i \), and each \( \lambda_i \) occurs with multiplicity \( \dim A_i \). In fact, $A_i$ can be identified with the endomorphisms of an irreducible representation of $G$. This is well known over $\mathbb{C}$ (cf. \cite[p.166]{noncommutativealgebra}), however, the argument applies verbatim to any algebraically closed field of characteristic zero.

Let \( \chi_j \) denote the irreducible character of \( A \) associated to the ideal \( A_j \). Then, evaluating \( \chi_j \) on \( z \) in two different ways gives  
\[
\chi_j(z) = \sum_{t \in S} x^{\beta(t)} \chi_j(t),
\]  
and also, since \( \chi_j(e_i) = 0 \) for \( i \ne j \) and \( \chi_j(e_j) = \chi_j(1) \),  
\[
\chi_j(z) = \sum_{i=1}^s \lambda_i \chi_j(e_i) = \lambda_j \chi_j(1).
\]  
Combining these expressions yields an explicit formula for \( \lambda_j \):  
\[
\lambda_j = \frac{\sum_{t \in S} x^{\beta(t)} \chi_j(t)}{\chi_j(1)}.
\]

Therefore, I find that 
\[
f_{\mathscr{X}, \alpha}(T) = \op{det}\left(r\cdot \op{Id}-\op{ad}\right)=\prod_{\chi\in \op{Irr}(G)} \op{det}(r\cdot \op{Id}-\lambda_j)^{\chi_j(1)^2}=\prod_{\chi\in \op{Irr}(G)} P_\chi(T)^{\chi(1)^2}.
\]

\end{proof}

The preceding result motivates the definition of an \emph{Iwasawa polynomial} associated to an irreducible representation of \( G \).

\begin{definition}
Let \( \rho: G \to \operatorname{GL}_d(\bar{\mathbb{Q}}_\ell) \) be an irreducible representation of \( G \), and let \( \chi = \operatorname{tr} \rho \) denote its character. Define the element
\[
P_{\chi}(T) := \frac{Q_{\chi}(T) }{\chi(1)}=\frac{r\chi(1) - \sum_{t \in S} (1 + T)^{\beta(t)} \chi(t)}{\chi(1)}.
\]
I call \( P_{\chi}(T) \) the \emph{Iwasawa function} attached to \( \chi \). It admits a factorization of the form
\[
Q_{\chi}(T) = \varpi^{\mu} f(T) u(T),
\]
where \( \mu \in \Z \), \( f(T) \in \mathcal{O}[T] \) is a distinguished polynomial, and \( u(T) \in \mathcal{O}\llbracket T \rrbracket^\times \) is a unit. I define
\[
\mu_\chi := \mu \quad \text{and} \quad \lambda_\chi := \deg f(T),
\]
and refer to \( \mu_\chi \) and \( \lambda_\chi \) as the \emph{\(\mu\)-invariant} and \emph{\(\lambda\)-invariant}, respectively, associated to \( \sX \), \( \chi \), and the function \( \beta: S \to \mathbb{Z}_\ell \).
\end{definition}

When \(\chi = 1\) is the character of the trivial representation, I obtain  
\[
P_1(T) = r - \sum_{t \in S} (1 + T)^{\beta(t)}.
\]  
I decompose the set \(S\) as a disjoint union  
\[
S = X \sqcup X^{-1} \sqcup X',
\]  
where \(X = \{h_1, \dots, h_k\}\) consists of elements \(h_i\) satisfying \(h_i^2 \neq 1\), and \(X' = \{h_{k+1}, \dots, h_m\}\) consists of involutions, i.e., elements \(h_i\) with \(h_i^2 = 1\). I set $\beta_i:=\beta(h_i)$. Note that $\beta(h_i^{-1})=-\beta(h_i)$ and thus if $h_i^2=1$, then, $\beta_i=\beta(h_i)=0$. Assume without loss of generality that for $i\leq k$, $\beta_i\geq 0$.
\begin{lemma}
    With respect to notation above, $T$ divides $P_1(T)$.
\end{lemma}
\begin{proof} Observe that
\[\begin{split}P_1(T)=&\sum_{i=1}^k\left(2-(1+T)^{\beta_i}-(1+T)^{-\beta_i}\right)\\ 
=& -\sum_{i=1}^k (1+T)^{-\beta_i}\left((1+T)^{\beta_i}-1\right)^2.
\end{split}\]
Thus, I see that $T$ divides $P_1(T)$.
\end{proof}

\begin{proposition}
With respect to notation above, one has that:
\[\mu_\ell(\sX, \alpha)=\frac{1}{e}\sum_{\chi\in \op{Irr}(G)} \chi(1)^2\mu_\chi\text{ and }\lambda_\ell(\sX, \alpha)=\sum_{\chi\in \op{Irr}(G)} \chi(1)^2\lambda_\chi -1,\] where $(\ell)=(\varpi^e)$ as ideals in $\cO$.
\end{proposition}

\begin{proof}
    Recall that $f_{\sX,\alpha}(T) = T g_{\sX,\alpha}(T)$ and therefore, by Theorem \ref{main thm decomp}
 \[g_{\sX, \alpha}(T)=P_1(T)/T\times \prod_{1\neq \chi\in \op{Irr}(G)} P_\chi(T)^{\chi(1)^2},\] from which the result follows easily.
\end{proof}

Next, I study the $\mu$ and $\lambda$-invariants of a character $\chi\in \op{Irr}(G)$. The first observation concerns a formula for the constant coefficient of $P_\chi(T)$.

\begin{lemma}
    Let $\chi\in \op{Irr}(G)$ and assume for simplicity that $\ell\nmid \chi(1)$. The following assertions hold:
\begin{enumerate}
    \item if $\chi\neq 1$, then,
    \[Q_\chi(0)=
    \sum_{t\in S}\left(\chi(1)-\chi(t)\right),\] and $\ell \nmid Q_\chi(0)$ if and only if $\mu_\chi=0$ and $\lambda_\chi=0$.
    \item If $\chi=1$, then, 
    \[Q_1'(0)=-\sum_{i=1}^k \beta_i^2, \]and $\ell \nmid Q_1'(0)$ if and only if $\mu_1=0$ and $\lambda_1=1$.
\end{enumerate}
\end{lemma}
\begin{proof}
    First, suppose that $\chi$ is nontrivial. In this case, the computation of $Q_\chi(0)$ is straightforward and left to the reader. Note that $\ell\nmid Q_\chi(0)$ if and only if $Q_\chi(T)$ is a unit in $\cO\llbracket T\rrbracket$ and this condition is equivalent to $\mu_\chi=0$ and $\lambda_\chi=0$. Next, consider the case when $\chi=1$, and in this case, since $T$ divides $Q_1(T)$, it follows that $\lambda_1\geq 1$. Note that 
    \[Q_1(T)=P_1(T)= -\sum_{i=1}^k (1+T)^{-\beta_i}\left((1+T)^{\beta_i}-1\right)^2\]
and therefore $Q_1'(0)=-\sum_{i=1}^k \beta_i^2$. From the Weierstrass preparation theorem, it is easy to see that $\mu_\chi=0$ and $\lambda_\chi=1$ if and only if $\ell\nmid Q_1'(0)$. 
\end{proof}

\begin{definition}
Let \(\rho_1, \rho_2: G \to \operatorname{GL}_n(\mathcal{O})\) be irreducible representations, and let \(\bar{\rho}_i: G \to \operatorname{GL}_n(\kappa)\) denote the reductions of \(\rho_i\) modulo \((\varpi)\). I say that \(\rho_1\) and \(\rho_2\) are \emph{congruent} if \(\bar{\rho}_1 \simeq \bar{\rho}_2\). Similarly, two characters \(\chi_1\) and \(\chi_2\) are said to be \emph{congruent} if \(\chi_1 \equiv \chi_2 \pmod{(\varpi)}\).
\end{definition}

\begin{proposition}\label{congruence}
    With respect to notation above, suppose that characters \(\chi_1\) and \(\chi_2\) are congruent and that $n:=\chi_i(1)$ is prime to $\ell$. Then it follows that 
\[\mu_{\chi_1}=0\Leftrightarrow \mu_{\chi_2}=0\]
and if the above conditions hold then $\lambda_{\chi_1}=\lambda_{\chi_2}$.
\end{proposition}
\begin{proof}If $\chi_1\equiv \chi_2\pmod{(\varpi)}$ then $Q_{\chi_1}\equiv Q_{\chi_2} \pmod{(\varpi)}$, and the result clearly follows.
\end{proof}

\subsection*{An example} Let me conclude with a concrete example. Let $\F_q$ be a finite field, and let $G := \mathrm{GL}_2(\F_q)$ be the group of invertible $2 \times 2$ matrices over $\F_q$. Fix a prime number $\ell$, and let $\cK$ be a sufficiently large finite extension of $\Q_\ell$ such that every irreducible representation $\rho : G \to \mathrm{GL}_n(\overline{\Q}_\ell)$ is defined over $\cK$.

Let me denote by $S := G \setminus \{\mathrm{Id}\}$ the set of all non-identity elements of $G$. I define a function $\beta : S \to \Z_\ell$ as follows. First, I partition the multiplicative group $\F_q^\times$ as
\[
\F_q^\times = \{a_1, \dots, a_k\} \cup \{a_1^{-1}, \dots, a_k^{-1}\} \cup \{a_{k+1}, \dots, a_\ell\},
\]
where the elements $a_i$ are chosen such that $a_i^2 \ne 1$ for $i \leq k$, and $a_i^2 = 1$ for $i > k$.

Now define $\beta : S \to \Z_\ell$ by the following rule:
\begin{itemize}
\item If $g \in S$ is not a scalar matrix, then $\beta(g) := 0$.
\item If $g = a_i \cdot \mathrm{Id}$ with $i \leq k$, then set $\beta(g) := 1$.
\item If $g = a_i^{-1} \cdot \mathrm{Id}$ with $i \leq k$, then set $\beta(g) := -1$.
\item If $g = a_i \cdot \mathrm{Id}$ with $i > k$, then set $\beta(g) := 0$.
\end{itemize}

It is straightforward to verify that the function $\beta$ satisfies the conditions of Definition~\ref{def of beta}. First, consider the natural permutation representation of $G$ on the projective line $\mathbb{P}^1(\F_q)$. This yields a representation of dimension $q+1$, which contains the trivial representation as a subrepresentation. Let $V$ be the unique complementary $q$-dimensional subrepresentation, and let $\chi_V$ be the character of $V$. Then for any scalar matrix $a \cdot \mathrm{Id} \in G$, I have
\[
\chi_V(a \cdot \mathrm{Id}) = q.
\]
Therefore, the corresponding polynomial becomes
\[
\begin{split}
P_{\chi_V}(T) &= (q - 2) - \sum_{1 \ne a \in \F_q^\times} (1+T)^{\beta(a \cdot \mathrm{Id})} \chi_V(a \cdot \mathrm{Id}) \\
&= (q - 2) - q \sum_{1 \ne a \in \F_q^\times} (1+T)^{\beta(a \cdot \mathrm{Id})}.
\end{split}
\]
I find that:
\[
\sum_{1 \ne a \in \F_q^\times} (1+T)^{\beta(a \cdot \mathrm{Id})} = k(1+T) + k(1+T)^{-1} + (q - 2 - 2k),
\]
and so
\[
\begin{split}
P_{\chi_V}(T) &= (q - 2) - \left[k(1+T) + k(1+T)^{-1} + (q - 2 - 2k)\right] \\
&= -T(1+T)(2+T)\cdot k.
\end{split}
\]

Consider the family of irreducible representations obtained by inducing characters from the Borel subgroup. Let $\alpha, \beta : \F_q^\times \to \cK^\times$ be two distinct characters. Let $B \subset G$ be the Borel subgroup consisting of upper triangular matrices
\[
B := \left\{ \begin{pmatrix} a & b \\ 0 & c \end{pmatrix} : a, c \in \F_q^\times, b \in \F_q \right\}.
\]
Define a character $\alpha \otimes \beta : B \to \cK^\times$ by
\[
(\alpha \otimes \beta)\begin{pmatrix} a & b \\ 0 & c \end{pmatrix} := \alpha(a)\beta(c).
\]
Let $W_{\alpha, \beta} := \mathrm{Ind}_B^G(\alpha \otimes \beta)$ be the induced representation. When $\alpha \ne \beta$, this representation is irreducible. Let $\chi_{\alpha, \beta}$ denote its character. Then for any scalar matrix $a \cdot \mathrm{Id} \in G$, one has
\[
\chi_{\alpha, \beta}(a \cdot \mathrm{Id}) = (q + 1)\alpha(a)\beta(a).
\]
Therefore, the associated polynomial becomes
\[
P_{\alpha, \beta}(T) = (q - 2) - \sum_{1 \ne a \in \F_q^\times} (1+T)^{\beta(a \cdot \mathrm{Id})} \cdot \alpha(a)\beta(a).
\]
\subsection*{Declarations}
\begin{description}
     \item[Conflict of interest] Not applicable, there is no conflict of interest to report.
    \item[Availability of data and materials] No data was generated or analyzed in obtaining the results in this article.
    \item[Funding] There are no funding sources to report.
\end{description}

\bibliographystyle{alpha}
\bibliography{references}

\begin{thebibliography}{HMSV24}

\bibitem[BN07]{BakerNorine}
Matthew Baker and Serguei Norine.
\newblock Riemann-{R}och and {A}bel-{J}acobi theory on a finite graph.
\newblock {\em Adv. Math.}, 215(2):766--788, 2007.

\bibitem[CP18]{divisorsandsandpiles}
Scott Corry and David Perkinson.
\newblock {\em Divisors and sandpiles}.
\newblock American Mathematical Society, Providence, RI, 2018.
\newblock An introduction to chip-firing.

\bibitem[DLRV24]{DLRV}
C\'{e}dric Dion, Antonio Lei, Anwesh Ray, and Daniel Valli\`eres.
\newblock On the distribution of {I}wasawa invariants associated to multigraphs.
\newblock {\em Nagoya Math. J.}, 253:48--90, 2024.

\bibitem[DV23]{DuBose/Vallieres:2022}
Sage DuBose and Daniel Valli\`eres.
\newblock On {$\Bbb Z^d_\ell$}-towers of graphs.
\newblock {\em Algebr. Comb.}, 6(5):1331--1346, 2023.

\bibitem[FD93]{noncommutativealgebra}
Benson Farb and R.~Keith Dennis.
\newblock {\em Noncommutative algebra}, volume 144 of {\em Grad. Texts Math.}
\newblock New York, NY: Springer-Verlag, 1993.

\bibitem[Gon21]{Gonet:2021a}
Sophia~R. Gonet.
\newblock {\em Jacobians of finite and infinite voltage covers of graphs}.
\newblock ProQuest LLC, Ann Arbor, MI, 2021.
\newblock Thesis (Ph.D.)--The University of Vermont and State Agricultural College.

\bibitem[Gon22]{Gonet:2022}
Sophia~R. Gonet.
\newblock Iwasawa theory of {J}acobians of graphs.
\newblock {\em Algebr. Comb.}, 5(5):827--848, 2022.

\bibitem[GR25]{GhoshRay}
Sohan Ghosh and Anwesh Ray.
\newblock On the {I}wasawa theory of {C}ayley graphs.
\newblock {\em Res. Math. Sci.}, 12(1):Paper No. 2, 20, 2025.

\bibitem[HMSV24]{HammerMattmanSandsVallieres}
Kyle Hammer, Thomas~W. Mattman, Jonathan~W. Sands, and Daniel Valli\`eres.
\newblock The special value {$u=1$} of {A}rtin-{I}hara {$L$}-functions.
\newblock {\em Proc. Amer. Math. Soc.}, 152(2):501--514, 2024.

\bibitem[KM22]{Kleine/Muller:2022}
S\"{o}ren Kleine and Katharina M\"{u}ller.
\newblock On the growth of the {J}acobian in $\mathbb{Z}_{p}^{l}$-voltage covers of graphs.
\newblock {\em {P}reprint, arxiv:2211.09763}, 2022.

\bibitem[LM24]{LeiMuller}
Antonio Lei and Katharina M\"{u}ller.
\newblock On the zeta functions of supersingular isogeny graphs and modular curves.
\newblock {\em Arch. Math. (Basel)}, 122(3):285--294, 2024.

\bibitem[MV23]{mcgownvallieresII}
Kevin McGown and Daniel Valli\`eres.
\newblock On abelian {$\ell$}-towers of multigraphs {II}.
\newblock {\em Ann. Math. Qu\'{e}.}, 47(2):461--473, 2023.

\bibitem[MV24]{mcgownvallieresIII}
Kevin McGown and Daniel Valli\`eres.
\newblock On abelian {$\ell$}-towers of multigraphs {III}.
\newblock {\em Ann. Math. Qu\'{e}.}, 48(1):1--19, 2024.

\bibitem[Nor98]{Northshield}
Sam Northshield.
\newblock A note on the zeta function of a graph.
\newblock {\em J. Combin. Theory Ser. B}, 74(2):408--410, 1998.

\bibitem[RV22]{Ray/Vallieres:2022}
Anwesh Ray and Daniel Valli{\`e}res.
\newblock An analogue of {K}ida's formula in graph theory.
\newblock {\em {P}reprint, arXiv:2209.04890}, 2022.

\bibitem[Ter11]{Terras:2011}
Audrey Terras.
\newblock {\em Zeta functions of graphs}, volume 128 of {\em Cambridge Studies in Advanced Mathematics}.
\newblock Cambridge University Press, Cambridge, 2011.
\newblock A stroll through the garden.

\bibitem[Val21]{Vallieres:2021}
Daniel Valli\`eres.
\newblock On abelian {$\ell$}-towers of multigraphs.
\newblock {\em Ann. Math. Qu\'{e}.}, 45(2):433--452, 2021.

\end{thebibliography}
\end{document}